\documentclass[11pt,a4paper]{amsart}
\usepackage{amsmath,amsfonts,amsthm,amssymb}
\usepackage[T1]{fontenc}
\usepackage{color}
\usepackage{epsfig}
\usepackage{graphicx}
\usepackage{cite,url}
\usepackage{hyperref}
\usepackage{xspace}
\usepackage{enumitem}
\usepackage[left=1in,right=1in,top=1in,bottom=1in,foot=0.5in]{geometry}
\usepackage{thm-restate}

\theoremstyle{plain}
\newtheorem{theorem}{Theorem}\newtheorem{lemma}[theorem]{Lemma}
\newtheorem{proposition}[theorem]{Proposition}

\theoremstyle{definition}

\newtheorem{remark}[theorem]{Remark}

\numberwithin{equation}{section}

\DeclareMathOperator{\Link}{Link}

\begin{document}

\title[Boundary rigidity of finite CAT(0) cube complexes]{Boundary rigidity of CAT(0) cube complexes} 

\author[J.\ Chalopin]{J\' er\'emie Chalopin}
\address{CNRS and Aix-Marseille Universit\'e, LIS, Marseille, France}
\email{jeremie.chalopin@lis-lab.fr}

\author[V.\ Chepoi]{Victor Chepoi}
\address{Aix-Marseille Universit\'e and CNRS, LIS, Marseille, France}
\email{victor.chepoi@lis-lab.fr}

\begin{abstract} In this note, we prove that finite CAT(0) cube complexes can be reconstructed from their boundary distances (computed in their 1-skeleta).
This result was conjectured by Haslegrave, Scott, Tamitegama, and Tan (2023). The reconstruction of a finite cell complex from the boundary distances is the discrete version of the boundary rigidity problem, which is a classical
problem from Riemannian geometry.  In the proof, we use the bijection between CAT(0) cube complexes and median graphs, and corner peelings of median graphs. \end{abstract}

\maketitle

\section{Introduction}
A natural question, arising in several research areas, is whether the
internal structure of an object can be determined from distances
between its boundary points. By a classical result in phylogeny by
Buneman~\cite{Bu} and Zareckii~\cite{Za}, trees can be reconstructed
from the pairwise distances between their leaves. In Riemannian
geometry, the notion of reconstruction from a distance function on the
boundary is well-established and related to boundary rigidity
questions. A Riemannian manifold $(M,g)$ is said to be \emph{boundary
  rigid} if its metric $d_g$ (defined on all pairs of points,
including the interior) is determined up to isometry by its boundary
distance function.  Michel~\cite{Mi} conjectured that any simple
compact Riemannian manifold with boundary is boundary rigid. The case
of 2-dimensional Riemannian manifolds was confirmed by Pestov and
Uhlmann~\cite{PeUh}. In higher dimensions, the conjecture is wide open and 
has been
established only in two cases by Besson, Courtois, and
Gallot~\cite{BeCoGa} and by Burago and Ivanov~\cite{BuIv}. 
The discrete version of
boundary rigidity was suggested by Benjamini (see~\cite{Ha,HaScTaTa})
who asked if any plane triangulation in which all inner vertices have degrees at least 6 can be
reconstructed from the distances between its boundary vertices. This
was answered in the affirmative by Haslegrave~\cite{Ha}, who also
proved a similar result for plane
quadrangulations in which all inner vertices have degrees at least 4. Haslegrave,
Scott, Tamitegama, and Tan~\cite{HaScTaTa} generalized the second
result of~\cite{Ha} and proved that any finite 2-dimensional CAT(0)
cube complex and any finite 3-dimensional CAT(0) cube complex $X$ with
an embedding in ${\mathbb R}^3$  can be
reconstructed from its boundary distances. They conjectured (see~\cite[Conjecture 20]{HaScTaTa}) that all finite CAT(0) cube complexes
can be reconstructed up to isomorphism from their boundary
distances. In the papers~\cite{Ha} and~\cite{HaScTaTa} the boundary $\partial X$
of $X$ is defined combinatorially (independently of the embedding of
$X$ in some ${\mathbb R}^k$), and the input information is the set of
all pairwise distances in the graph (1-skeleton) $G(X)$ of the complex
$X$ between all vertices of $G(X)$ belonging to $\partial X$.  In this
note, we confirm the conjecture of~\cite{HaScTaTa} and prove the following theorem.

\begin{restatable}{theorem}{thmCATBoundaryRigid}\label{thm-CAT0-boundary-rigid}
  Let $X$ be a finite CAT(0) cube complex. Then $X$ is determined up
  to isomorphism by the distances between the vertices of $\partial X$
  in the 1-skeleton of $X$. Consequently, the class of finite CAT(0)
  cube complexes is boundary rigid.
\end{restatable}

In the proof, we use the bijection between  CAT(0) cube complexes and median graphs 
and corner peelings of median graphs. CAT(0) cube complexes are central objects in geometric group theory;
see Sageev~\cite{Sa,Sa_survey} and Wise~\cite{Wi}. They have been
characterized by Gromov~\cite{Gr} as simply connected cube complexes
in which the links of vertices are simplicial flag complexes. The
systematic investigation and use of CAT(0) cube complexes in geometric
group theory started with the paper by Sageev~\cite{Sa}. Soon after,
it was proved by Chepoi~\cite{Ch_CAT} and Roller~\cite{Ro} that
1-skeleta of CAT(0) cube complexes are exactly the median
graphs. Median graphs are central objects in metric graph theory; they were introduced by Nebesk{\'y}~\cite{Ne} 
and investigated in numerous papers, see for example
the papers by Mulder~\cite{Mu,Mu_exp} and the survey by Bandelt and Chepoi~\cite{BaCh_survey}.

\section{Preliminaries}
In this section, we recall the boundary rigidity problem for cell complexes. We also recall some definitions and basic facts about CAT(0) cube complexes
and median graphs. 

\subsection{Graphs}
All graphs $G=(V(G),E(G))$ considered in this paper are finite, undirected,
connected, and contain neither multiple edges, nor loops.  For two distinct
vertices $v,w\in V(G)$ we write $v\sim w$ when there is an edge
connecting $v$ with
$w$.  The subgraph
of $G$ \emph{induced by} a subset $A\subseteq V(G)$ is the graph
$G[A]=(A,E')$ such that $uv\in E'$ if and only if $uv\in E(G)$.
A \emph{square} $uvwz$
is an induced $4$--cycle $(u,v,w,z)$. Equivalently, a square is a graph that is isomorphic to the 1-skeleton of a  unit Euclidean square in the plane.  More generally, a \emph{cube} $Q_n$ of dimension $n$ is a graph
isomorphic to the 1-skeleton of the $n$-dimensional unit Euclidean
cube. Alternatively, $Q_n$ is a graph whose vertices
can be labeled by the subsets of a set of size $n$ and such that
two vertices are adjacent if and only if the corresponding sets $A,B$  differ in a single element, i.e., if $|A\triangle B|=1$.
The \emph{distance}
$d(u,v)=d_G(u,v)$ between two vertices $u$ and $v$ of a graph $G$ is the
length of a shortest $(u,v)$--path.  An induced subgraph $H=G[A]$ of a graph $G$ is an \emph{isometric subgraph} of $G$ if $d_H(u,v)=d_G(u,v)$ for any
two vertices $u,v\in A$.
The \emph{interval}
$I(u,v)$ between $u$ and $v$ consists of all vertices on shortest
$(u,v)$--paths, that is, of all vertices (metrically) \emph{between} $u$
and $v$: $I(u,v)=\{ x\in V(G): d(u,x)+d(x,v)=d(u,v)\}$. An induced
subgraph of $G$ is called \emph{convex}
if it includes the interval of $G$ between any pair of its
vertices.
An induced subgraph $H$ (or the corresponding vertex set of $H$)
of a graph $G$
is \emph{gated}  if for every vertex $x$ outside $H$ there
exists a vertex $x'$ in $H$ (the \emph{gate} of $x$)
such that  $x'\in I(x,y)$ for each $y$ of $H$. Gated sets are convex and
the intersection of two gated sets is
gated.

\subsection{Cube complexes and boundary rigidity} All cube complexes considered in this
paper are finite cell complexes whose cells are unit Euclidean cubes and whose gluing maps are isometries (see the book of Hatcher~\cite{Hat}).
The \emph{dimension} $\dim(C)$ of a cell $C$ is the dimension of
the cube $C$. The \emph{dimension} $\dim(X)$ of $X$ is the largest dimension of
a cell of $X$.  If $C$ is a cell of $X$ of dimension $k$, then all
cells of $X$ contained in $C$ and having dimension $k-1$ are called
the \emph{facets} of $C$. Cells of $X$ are called \emph{maximal} if
they are maximal by inclusion, and \emph{non-maximal} otherwise.
For a cube complex $X$, we denote by $X^{(k)}$ the \emph{$k$--skeleton} of $X$ consisting of all cubes of dimension at most $k$.
We use the notations $V(X) = X^{(0)}$ for the set of vertices
(0-cubes) of $X$ and $G(X):=X^{(1)}$ for the 1-skeleton of $X$; the
graph $G(X)$ will be always endowed with the standard graph distance
$d_G$. An \emph{abstract simplicial complex} $\Delta$ on a finite set
$V$ is a set of nonempty subsets of $V$, called \emph{simplices}, such
that any nonempty subset of a simplex is also a simplex.
The \emph{clique complex} of a graph $G$ is the abstract simplicial
complex having the cliques (i.e., complete subgraphs) of $G$ as
simplices. A simplicial complex $X$ is a \emph{flag simplicial
  complex} if $X$ is the clique complex of its $1$--skeleton.  In a
cube complex $X$, the \emph{link} $\Link(x)$ of a vertex $x$ is the
simplicial complex whose vertices are the edges of $X$ containing $x$
and whose simplices are given by the collections of edges belonging to
a common cube of X.

We will use a slightly weaker version of the definition of
combinatorial boundary given by Haslegrave et al.~\cite{HaScTaTa}.
For a finite cube complex $X$, the \emph{combinatorial boundary}
$\partial X$ is the downward closure of all non-maximal cells of $X$
such that each of them is a facet of a unique cell of $X$.  The
definition of~\cite{HaScTaTa} is similar but instead of maximality by
inclusion the maximality by dimension is considered:
in~\cite{HaScTaTa}, the combinatorial boundary of a cube complex $X$
of dimension $k$ is the downward closure of all cells of $X$ of
dimension less than $k$ belonging to at most one cell of dimension
$k$. Obviously, the boundary of a cube complex $X$ defined in our way
is always contained in the boundary of $X$ defined as
in~\cite{HaScTaTa}. The two boundaries are different as soon as the
cube complex $X$ contains maximal cells of different dimensions.

The \emph{boundary distance matrix} of a finite cube complex $X$ is the
matrix whose rows and columns are the vertices of $\partial X$ and
whose entries are the distances in $G(X)$ between the corresponding
vertices of $\partial X$.  A class $\mathfrak{C}$ of finite cube
complexes is called \emph{boundary rigid} if for any two complexes
$X,Y \in \mathfrak{C}$ that have the same boundary distance matrices
up to a permutation $\sigma$ of the rows and the columns ($\sigma$ can
be seen as an isometry between $\partial X$ and $\partial Y$), then
$X$ and $Y$ are isomorphic via an isomorphism $\widetilde{\sigma}$
extending $\sigma$.

\subsection{CAT(0) cube complexes and median graphs}
In this subsection, we recall the definitions of CAT(0) cube complexes and median graphs, and the bijection between them.

Let $(X,d)$ be a metric space.
A \emph{geodesic} joining two
points $x$ and $y$ from $X$ is a map $\rho$ from the segment $[a,b]$
of ${\mathbb R}^1$ of length $|a-b|=d(x,y)$ to $X$ such that
$\rho(a)=x, \rho(b)=y$, and $d(\rho(s),\rho(t))=|s-t|$ for all $s,t\in
[a,b]$. A metric space $(X,d)$ is \emph{geodesic} if every pair of
points in $X$ can be joined by a geodesic.
A \emph{geodesic triangle} $\Delta=\Delta (x_1,x_2,x_3)$ in a geodesic
metric space $(X,d)$ is defined by three points in $X$ (the vertices
of $\Delta$) and three arbitrary geodesics (the sides of $\Delta$),
one between each pair of vertices $x_1, x_2, x_3$. A \emph{comparison
  triangle} for $\Delta (x_1,x_2,x_3)$ is a triangle
$\Delta (x'_1,x'_2,x'_3)$ in the Euclidean plane ${\mathbb E}^2$ such
that $d_{{\mathbb E}^2}(x'_i,x'_j)=d(x_i,x_j)$ for
$i,j\in \{ 1,2,3\}$. A geodesic metric space $(X,d)$ is a \emph{CAT(0)
  space}~\cite{Gr} if all geodesic triangles $\Delta (x_1,x_2,x_3)$ of
$X$ satisfy the comparison axiom of Cartan--Alexandrov--Toponogov:
\emph{If $y$ is a point on the side of $\Delta(x_1,x_2,x_3)$ with
  vertices $x_1$ and $x_2$ and $y'$ is the unique point on the line
  segment $[x'_1,x'_2]$ of the comparison triangle
  $\Delta(x'_1,x'_2,x'_3)$ such that
  $d_{{\mathbb E}^2}(x'_i,y')= d(x_i,y)$ for $i=1,2$, then
  $d(x_3,y)\le d_{{\mathbb E}^2}(x'_3,y')$.} A geodesic metric space
$(X,d)$ is \emph{nonpositively curved} if it is locally CAT(0),
i.e., any point has a neighborhood inside which the CAT(0)
inequality holds.  CAT(0) spaces can be characterized in several
different natural ways and have many strong properties, see for
example the book by Bridson and Haefliger~\cite{BrHa}.  In particular, a geodesic metric space $(X,d)$
is CAT(0) if and only if $(X,d)$ is simply connected and is
nonpositively curved. For CAT(0) cube complexes, Gromov~\cite{Gr} gave
a beautiful combinatorial analog of this result, which can be also
taken as their definition:

\begin{theorem}[\!\!\cite{Gr}]\label{Gromov} A cube complex $X$ endowed  with the $\ell_2$-metric is CAT(0) if and
only if $X$ is simply connected and the links of all vertices of $X$ are flag simplicial complexes.
\end{theorem}

A graph $G$ is called \emph{median} if the interval intersection
$I(x,y)\cap I(y,z)\cap I(z,x)$ is a singleton for each triplet $x,y,z$ of vertices. Median graphs are bipartite.
Basic examples of median graphs are trees, hypercubes, rectangular grids, and
Hasse diagrams of distributive lattices and  of median semilattices~\cite{BaCh_survey}.
Any two cubes of a median graphs are either disjoint or intersect in a cube. Therefore, replacing each cube of $G$
by a solid unit cube of the same dimension, we  obtain a cube complex $X_{cube}(G)$.
We continue with the bijection between CAT(0) cube complexes and median graphs:

\begin{theorem}[\!\!\cite{Ch_CAT,Ro}]\label{CAT(0)=median} Median graphs are exactly the 1-skeleta  of CAT(0) cube complexes.
\end{theorem}

The proof of Theorem~\ref{CAT(0)=median} presented in~\cite{Ch_CAT} is fact establishes  the following local-to-global characterization of median graphs:

\begin{theorem}[\!\!\cite{Ch_CAT}]\label{CAT(0)=median1} A graph $G$ is a median graph if and only if its cube complex $X_{cube}(G)$ is simply connected
and $G$ satisfies the 3-cube condition: if three squares of $G$ pairwise intersect in an edge and all three intersect in a vertex,
then they belong to a 3-cube. Furthermore, if $X$ is a CAT(0) cube complex, then $X=X_{cube}(G(X))$, i.e., $X$ can be retrieved from its 1-skeleton $G(X)$ by replacing each graphic cube
of $G(X)$ by a solid cube.
\end{theorem}

\section{Boundary rigidity of CAT(0) cube complexes}

In this section, we describe a simple method to reconstruct  a CAT(0) cube complex $X$ from boundary distances.
By Theorem~\ref{CAT(0)=median1}, it suffices to reconstruct the 1-skeleton $G(X)$ of $X$, because
we know that $X=X_{cube}(G)$. Our method of reconstructing $G(X)$
uses the fact that median graphs admit corner peelings; see the paper by Chalopin, Chepoi, Moran and Warmuth~\cite{ChChMoWa}.
Let $X$ be a CAT(0) cube complex, $\partial X$ be its combinatorial boundary, and
$G(X)$ be the 1-skeleton of $X$. For simplicity, we will denote $G(X)$ by $G$. We denote by $\partial G\subseteq V(X)$ the vertices (0-cubes) of $X$ belonging to $\partial X$.
Recall that the input of the reconstruction problem is the boundary $\partial G$ and the distance matrix $D={(d(x,y))}_{x,y\in \partial G}$, where $d(x,y)$ is computed in $G$.
Therefore, two vertices $x,y$ of $\partial G$ are adjacent in $G$ if and only if $d(x,y)=1$.

\subsection{Facts about median graphs}\label{sec:properties}		
We recall the main properties of median graphs used in the
reconstruction. These results can be found in the papers by Mulder~\cite{Mu,Mu_exp} and are now a
part of folklore for people working in metric graph theory. From now on, $G$ is a finite median graph.  The first
property follows from the definition.
	
\begin{lemma}[Quadrangle Condition]\label{quadrangle}
  For any vertices $u,v,w,z$ of $G$ such that $v$ and $w$ are adjacent
  to $z$ and $d(u,v) = d(u,w) = d(u,z)-1 = k$, there is a unique
  vertex $x$ adjacent to $v$ and $w$ such that $d(u,x) = k-1$ (see
  Fig.~\ref{fig-qc-cond}).
\end{lemma}

\begin{figure}[h]
\begin{center}
\includegraphics{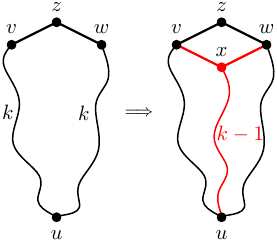}
\end{center}
\caption{Quadrangle condition}\label{fig-qc-cond}
\end{figure}

The next lemma follows from the fact that convex subgraphs of median graphs are gated and that cubes are convex subgraphs:

\begin{lemma}[Cubes are gated]\label{convex-gated}
Cubes of median graphs are gated.
\end{lemma}

Let $z$ be a basepoint of a median graph $G$ ($z$ is a fixed but arbitrary vertex of $G$). For a vertex $v$, $\Lambda(v)$
consists of all neighbors of $v$ in the interval $I(z,v)$. A graph $G$ satisfies the \emph{downward cube property with respect to $z$} if for each vertex
$v$, all its neighbors from $\Lambda(v)$ together with $v$ itself belong to a unique cube $C(v)$ of
$G$. In this case, we denote  by $\overline{v}$ the unique vertex
of the cube $C(v)$ opposite to $v$.  That median graphs satisfy the downward property was proved in~\cite{Mu}.  We provide a simple proof of this result from the paper by B\'en\'eteau, Chalopin, Chepoi, and Vax\`es~\cite{BeChChVa}:
	
\begin{lemma}[\!\!\cite{Mu})(Downward Cube Property]\label{descendent_cube}
  Let $G$ be a median graph and fix an arbitrary basepoint $z$ of
  $G$. Then $G$ satisfies the downward cube property with respect to
  $z$. Moreover, for each vertex $v$ of $G$, the vertex $\overline{v}$
  is the gate of $z$ in the cube $C(v)$.
\end{lemma}

\begin{proof} Fix a basepoint $z$, pick any vertex $v$, and let
  $\Lambda(v)=\{ u_1,\ldots,u_d\}$.
We prove that $v$ and any subset of $k$ vertices of $\Lambda (v)$
  belong to a unique cube of dimension $k$. We prove this result by
  induction on $d(z,v)$ and on the number $k$ of chosen neighbors.
  The assertion holds in the base case $d(z,v)=2$: since median
  graphs are $K_{2,3}$-free, $z$ and $v$ have one or two common
  neighbors. The assertion also holds in another base case when
  $k=2$: by the quadrangle condition, the two neighbors of $v$ in
  $\Lambda (v)$ have a unique common neighbor one step closer to $z$,
  yielding a 2-cube. Now pick $v$ and the $k\ge 3$ neighbors $u_1,\ldots,u_k$ and assume
  that $v$ and any $2\le k'<k$ neighbors of $v$ in $\Lambda (v)$
  belong to a unique $k'$-cube. By the quadrangle condition, for each
  $2 \leq i \leq k$, there exist a unique common neighbor $z_i$ of
  $u_1$ and $u_i$ one step closer to $z$.  Since median graphs are
  $K_{2,3}$-free, the vertices $z_2,\ldots,z_k$ are distinct neighbors
  of $u_1$ in $\Lambda(u_1)$. By induction hypothesis, $u_1$ and
  $z_2,\ldots,z_k$ belong to a unique $(k-1)$-cube $R'$. Analogously,
  $v$ and its neighbors $u_2,\ldots, u_k\in \Lambda (v)$ belong to a
  unique $(k-1)$-cube $R''$. We assert that the subgraph of $G$
  induced by the vertices of $R'\cup R''$ is a $k$-cube of $G$. For
  each $i=2,\dots, k$, the set $\Lambda(u_i)$ consists of the vertex
  $z_i$ and the $k-2$ neighbors in $R''$ other than $v$. By induction
  hypothesis, $u_i$, $z_i$, and the neighbors of $u_i$ in $R''$ other
  than $v$ define a $(k-1)$-cube $R_i$ of $G$. This implies that there
  exists an isomorphism between the facet of $R'$ containing $z_i$ and
  not $u_1$ and the facet of $R''$ containing $u_i$ and not containing
  $v$. Since $R'$ and $R''$ are gated and thus convex, this defines an
  isomorphism between $R'$ and $R''$ which maps vertices of $R'$ to
  their neighbors in $R''$. Hence $R'\cup R''$ defines a $k$-cube of
  $G$.  The unicity of $R'$ and $R''$ ensures that $R'\cup R''$ is the
  unique $k$-cube of $G$ containing $v$ and $u_1, \ldots, u_k$,
  finishing the proof of the first assertion. Since all vertices of
  $\Lambda(v)$ are closer to $z$ than $v$ and the cube $C(v)$ is
  gated, the gate of $z$ in $C(v)$ is necessarily the vertex
  $\overline{v}$.
\end{proof}

\subsection{Corner peelings of median graphs}

A \emph{corner} of a graph $G$ is a vertex $v$ of $G$ such that $v$ and all its neighbors in $G$ belong to a unique cube of $G$.
A \emph{corner peeling} of a finite graph $G$ is a total order $v_1,\ldots,v_n$ of $V(G)$ such that $v_i$ is a corner of the subgraph $G_i=G[v_1,\ldots,v_i]$ induced by the first $i$ vertices of this order.
If $v_1,\ldots, v_n$ is a corner peeling of $G$, then each level-graph $G_i$ ($i=n,\ldots,1$) is an isometric subgraph of $G$. A \emph{monotone corner peeling} of $G$ with respect to a basepoint $z$ is a corner peeling $v_1=z,v_2,\ldots,v_n$ such that $d(z,v_1)\le d(z,v_2)\le\cdots\le d(z,v_n)$. That median graphs admit monotone
corner peelings follows from~\cite[Proposition 4.9]{ChChMoWa} (which was proved in a more general setting but with respect to a particular basepoint). We need a stronger version  of this result:

\begin{proposition}\label{corner-peeling-median} For any finite median graph $G$ and any basepoint $z$, any ordering $v_1=z,v_2,\ldots,v_n$ of $V(G)$
such that $d(z,v_1)\le d(z,v_2)\le\cdots\le d(z,v_n)$ is a monotone corner peeling of $G$. Furthermore, $C(v_i)$ is the unique cube of $G_i$ containing $v_i$ and the neighbors of $v_i$ in $G_i$. The vertex $\overline{v_i}$
opposite to $v_i$ in $C(v_i)$ is the gate of $z$ in $C_i$ (in $G$ and in $G_i$).
\end{proposition}

\begin{proof}  Pick any vertex $v_i$. Since $d(z,v_1)\le d(z,v_2)\le\cdots\le d(z,v_i)\le\cdots\le d(z,v_n)$ holds, for any neighbor $v_j$ of $v_i$ in $G_i$ we have $d(z,v_j)\le d(z,v_i)$. Since $G$ is bipartite, we must have $d(z,v_j)<d(z,v_i)$. Consequently, all neighbors of $v_i$ in $G_i$ belong to the interval $I(z,v_i)$ and thus to the set $\Lambda(v_i)$. Conversely, for any vertex $v_j$ of $\Lambda(v_i)$ we have
$d(z,v_j)<d(z,v_i)$, thus $v_j$ belongs to $G_i$. Consequently, the set of neighbors of $v_i$ in $G_i$ coincides with $\Lambda (v_i)$. By
Lemma~\ref{descendent_cube}, $v_i$ and the set  $\Lambda(v_i)$ of all neighbors of $v_i$ in $G_i$ belong to a unique cube $C(v_i)$ of $G$. We assert that $C(v_i)$ is a cube of $G_i$.
Indeed, by the second assertion of Lemma~\ref{descendent_cube} the vertex $\overline{v_i}$ opposite to $v_i$ in $C(v_i)$ is the gate of $z$ in $C(v_i)$. Since $C(v_i)$ is the cube induced by $I(v_i,\overline{v_i})$ and $d(v_i,z)=d(v_i,\overline{v_i})+d(\overline{v_i},z)$,
for any vertex $v_j\in C(v_i)$ other than $v_i$ we obtain that $d(v_j,z)<d(v_i,z)$. Consequently, all vertices of $C(v_i)$ belong to  $G_i$,
hence  $v_i$  is a corner of $G_i$ and $C(v_i)$ is the unique cube of $G_i$ containing $v_i$ and its neighbors in $G_i$. The last assertion now follows from the second assertion of Lemma~\ref{descendent_cube}. \end{proof}

\begin{remark}
  The level-graphs $G_i$ ($i=n-1,\ldots, 1)$ are not necessarily
  median graphs.  Notice also that since the cube complexes of graphs
  with a corner peeling are collapsible, this provide an alternative
  proof of the result of Adiprasito and Benedetti~\cite{AdBe} that
  CAT(0) cube complexes are collapsible.
\end{remark}

Let $z$ be an arbitrary fixed vertex of $\partial G$ and let $v_1=z,v_2,\ldots,v_n$ be a monotone corner peeling of $G$ defined as in Proposition~\ref{corner-peeling-median}. 
For all $i=n,\ldots,1$ denote by $\partial G_i$ the set of vertices belonging to the boundary of the cube complex $X_i=X_{cube}(G_i)$.
Notice that $G_n=G$ and thus $\partial G_n=\partial G$. As in Proposition~\ref{corner-peeling-median},
for each vertex $v_i$ we denote by $C_i=C(v_i)$ the unique cube of $G_i$ containing the vertex $v_i$ and the set $\Lambda(v_i)$ of all neighbors of $v_i$ in $G_i$.

Let $u_i$ denote the opposite vertex $u_i=\overline{v_i}$ of $v_i$ in the cube $C_i$. From the second assertion of  Proposition~\ref{corner-peeling-median} we know that $u_i$ is the gate of
$z$ in the cube $C_i$.  Then $V(C_i)=I(v_i,u_i)$, where the interval $I(v_i,u_i)$ is considered in $G$ or in $G_i$.

From the definition of a corner it immediately follows that all
corners of a median graph $G$ belong to the boundary $\partial G$. In
fact, the following lemma shows that for any corner $v$ of $G$, any
vertex of $C(v)$ belongs to $\partial G$ except possibly
$\overline{v}$.

\begin{lemma}\label{boundary_corner} All vertices of the cube $C_i$ belong to the boundary $\partial G_i$ of $G_i$, except  possibly $u_i$.
\end{lemma}

\begin{proof} Since in $G_i$ all neighbors of $v_i$ belong to the cube
  $C_i$, any facet of $C_i$ containing $v_i$ is not properly contained
  in any other cube of $G_i$. Since all vertices on such facets belong
  to $\partial G_i$ and any vertex of $C_i$ except $u_i$ belong to
  such a facet, we conclude that every vertex of $C_i$ different from
  $u_i$ belongs to $\partial G_i$.
\end{proof}

Set $S(G_n)$ to be the set $\partial G_n=\partial G$ and iteratively define  the set $S(G_{i-1})$ to be $S(G_i) \setminus \{ v_i\} \cup \{ u_i\}$ $(i=n-1,\ldots, 2)$. We will call $S(G_i)$ the \emph{extended boundary} of $G_i$; the name is justified by the following result:

\begin{lemma}\label{boundary_corner2} For all $i=n,\ldots, 2$, we have $\partial G_{i-1} \subseteq \partial G_{i} \cup \{u_i\}$
  and $\partial G_i \subseteq S(G_i)$.
\end{lemma}

\begin{proof}
  We first prove that any vertex from $\partial G_{i-1}$ belongs to
  $\partial G_i\cup \{ u_i\}$. Suppose not and let
  $x\in \partial G_{i-1}\setminus \partial G_i$. Since
  $x\in \partial G_{i-1}$, there exists a non-maximal cube $C$
  containing $x$ and contained as a facet in a unique cube $C'$ of
  $X_{i-1}=X_{cube}(G_{i-1})$.  On the other hand, since
  $X_{i-1}\subset X_i$ and $x$ is not a boundary vertex of $X_i$, the cube $C$ is a facet of yet another cube $C''$ in $X_i$.  The cube
  $C''$ of $X_i$ is not present in $X_{i-1}$, and consequently $v_i$
  is a vertex of $C''$. Since all cubes of $G_i$ containing $v_i$ are
  included in $C_i$, we know that $x$ is a vertex of $C_i$. From
  Lemma~\ref{boundary_corner} and since $x \notin \partial G_i$,
  necessarily $x = u_i$. This establishes the inclusion
  $\partial G_{i-1} \subseteq \partial G_{i} \cup \{u_i\}$.

  We now prove that $\partial G_{i} \subseteq S(G_i)$ by decreasing
  induction on $i=n,\ldots,1$. For $i=n$, we have
  $S(G_n)=\partial G_n$. Suppose now that the assertion holds for
  $G_i$ and consider the graph $G_{i-1}$. Since $v_i \notin G_i$, by
  the first assertion of the lemma, and by induction hypothesis, we
  have
  $\partial G_{i-1} \subseteq \partial G_i \setminus \{v_i\} \cup
  \{u_i\} \subseteq S(G_i) \setminus \{v_i\} \cup \{u_i\} =
  S(G_{i-1})$.
\end{proof}

\subsection{Reconstruction via corner peeling}
To reconstruct  a median graph $G$ from the pairwise
distances between the vertices of the boundary $\partial G$ we proceed
in the following way.
We first describe the objects that the reconstructor keeps track of
during the reconstruction algorithm. It first picks an arbitrary
vertex $z \in \partial G$ as a basepoint that is fixed during the
whole execution of the algorithm.  During the execution of the
algorithm, the reconstructor knows a set $S$ of vertices that is
initially $\partial G$ as well as the distance matrix $D$ between the
vertices of $S$ computed in the graph $G$ (if $u,v \in S$, then
$D[u,v]$ contains the distance $d_G(u,v)$).  It constructs a graph
$\Gamma$ that is initially isomorphic to the subgraph of $G$ induced
by the boundary $\partial G$ and will ultimately coincide with $G$.

In order to analyze the algorithm, we consider the intermediate values
$S_i$ of the set $S$, $D_i$ of the distance matrix $D$, and $\Gamma_i$
of the graph $\Gamma$ at the beginning of the $i$th step of the
algorihm. In order to have similar notations as in the definition of
monotone corner peelings, we denote the initial values of $S$, $D$,
and $\Gamma$ respectively by $S_n$, $D_n$, and $\Gamma_n$ (even if the
algorithm does not know $n$), and at each step, we decrease the values
of $i$.  For the analysis of the algorithm, we also consider the graphs
$G_i$ (unknown to the algorithm), where $G_n = G$.

We now give an outline of the reconstruction algorithm.  The goal is
to reconstruct the graph $G = G_n$ from the set $S_n=\partial G$ and
its distance matrix $D_n$.  The graph $\Gamma_n$ induced by
$\partial G$ can be computed from $D_n$: two vertices $u$ and $v$ of
$S_n$ are adjacent in $\Gamma_n$ if and only if $D_n[u,v] = 1$.  At
step $i$, the reconstructor picks a vertex $v_i$ of $S_i$ furthest
from the basepoint $z$, identifies the vertex $u_i$ opposite to $v_i$
in the unique cube $C_i$ of $G$ containing $v_i$ and its neighbors in
$S_i$, removes $v_i$ from $S_i$, and adds $u_i$ to $S_i$ unless it is
already in $S_i$. The resulting set is denoted by $S_{i-1}$. From
$D_i$, we compute the distance matrix $D_{i-1}$ between the vertices
of $S_{i-1}$; the main point is to compute the distances from $u_i$ to
the vertices of $S_{i-1}$ since the other distances are already known.
If $u_i \in S_i$, then $\Gamma_{i-1}$ is $\Gamma_{i}$. Otherwise,
$\Gamma_{i-1}$ is obtained by adding to $\Gamma_i$ the vertex $u_i$
and the edges between $u_i$ and its neighbors in
$S_{i-1}\cup \{v_i\}$.
The algorithm ends when $S_i$ becomes empty.  We will show that the
graph $\Gamma_0$ obtained when $S_i = \varnothing$ is isomorphic to
$G$.

Now, we present the invariants that are maintained at each step $i$ of the
reconstruction algorithm.
Let $G_i$ be the subgraph of $G$ obtained from $G$ by removing the
vertices $v_n,\ldots, v_{ i+1}$. Notice that the subgraph $G_i$ is not
known to the reconstructor. Furthermore, suppose that the removed
vertices $v_n,\ldots, v_{i+1}$ and the possibly added vertices
$u_n,\ldots, u_{i+1}$ satisfy the following inductive properties:

\begin{enumerate}
\item\label{inv-1} $d(z,v_n)\ge \cdots \ge d(z,v_{i+1})\ge d(z,v)$ for any vertex of $v$
of $G_i$.
\item\label{inv-2}  Each vertex $v_j$ with $n\ge j\ge i+1$ is a corner of
the graph $G_j$. \item\label{inv-3} For each $n\ge j\ge i+1$, either all neighbors of
  $v_j$ in $G_j$ are in $S_j$ or $u_j$ is the unique neighbor of
  $v_j$ in $G_j$; in both cases, $u_j \in S_{j-1}$.
\item\label{inv-4}  $S_i$ coincides with the extended boundary
$S(G_i)$ of $G_i$, $D_i$ is the distance matrix of $S(G_i)$ in $G$,
and $\Gamma_i = G[\bigcup_{n \ge j \ge i} S_j]$.

\end{enumerate}

The next result explains how to find a corner of the graph $G_i$
(without knowing $G_i$) from the distances from $z$ to the vertices of
$S_i=S(G_i)$.

\begin{lemma}\label{furthest-corner}
  Let $v_i$ be a vertex of $S_i$ maximizing $d(z,v_i)$.
Then $d(z,v_i)\ge d(z,v)$ for any vertex $v$ of $G_i$ and thus $v_i$
  is a corner of $G_i$.
\end{lemma}

\begin{proof} Suppose by way of contradiction that there exists a
  vertex $u$ of $G_i$ such that $d(z,v_i)<d(z,u)$ and assume without
  loss of generality that $u$ maximizes $d(z,u)$ among all vertices in
  $G_i$.  Since $d(z,v_n)\ge \cdots \ge d(z,v_{i+1})\ge d(z,v)$ for
  any vertex of $v$ of $G_i$ by invariant (\ref{inv-1}), from
  Proposition~\ref{corner-peeling-median} we conclude that there
  exists a monotone corner peeling of $G$ starting with the vertices
  $v_n,\ldots, v_{i+1},u$.  Consequently, $u$ is a corner of $G_i$ and
  therefore $u$ belongs to $\partial G_i$. Since
  $\partial G_i\subseteq S(G_i)$ by Lemma~\ref{boundary_corner2} and
  $S(G_i)=S_i$ by invariant (\ref{inv-4}), the vertex $u$ belongs to
  $S_i$, contradicting the definition of $v_i$.  Consequently, $v_i$
  is a vertex of $G_i$ maximizing $d(z,v_i)$ and hence a corner of
  $G_i$.
\end{proof}

Suppose that we have performed the steps $n, \ldots, i+1$ of the
algorithm. Let $v_i$ be a vertex of $S_i=S(G_i)$ maximizing $d(z,v_i)$,
and assume that $v_i \neq z$.
Since $v_i$ is a corner of $G_i$ by Lemma~\ref{furthest-corner}, $v_i$
and its neighbors in $G_i$ belong to a unique cube $C_i$ of
$G_i$. The dimension of $C_i$ is the number of neighbors
of $v_i$ in $G_i$.  Let $u_i$ be the vertex of $C_i$ opposite to
$v_i$. By Lemmas~\ref{boundary_corner} and~\ref{boundary_corner2} and
invariant (\ref{inv-4}),
$V(C_i)\setminus \{ u_i\}\subseteq \partial G_i \subseteq S(G_i) =
S_i$. If $u_i \notin S_i$, then $u_i$ will be added
to $S_{i-1}$.

Additionally, the distances from $u_i$ to all other vertices $x$ of
$S_i\setminus \{v_i\}$ should be computed.  If $\dim(C_i)=1$, then
$u_i$ is an articulation point separating $v_i$ from all other
vertices of $G_i$, and thus $d(u_i,x) = d(v_i,x) -1$.  Now suppose
that $\dim(C_i) \geq 2$ and let $N(u_i)$ be the set of neighbors of
$u_i$ in $C_i$. Since $C_i$ is gated in $G$, we can consider the gate
$x'$ of $x$ in $C_i$. We assert that $x' = u_i$ if and only if all
vertices of $N(u_i)$ are at the same distance from $x$. Indeed, if
$x' = u_i$, then all vertices of $N(u_i)$ are at distance $d(u_i,x)+1$
from $x$. Conversely, if all vertices of $N(u_i)$ are at the same
distance from $x$, then either $x' = u_i$ or $x'=v_i$. However, the
second case is impossible because $v_i$ is a corner of $G_i$ and
$x \in S_i\setminus \{v_i\} \subseteq V(G_i) \setminus \{v_i\}$.
Suppose now that $x' \neq u_i$ (for an illustration, see
Figure~\ref{fig-algo}). By the previous assertion, there exist
vertices in $N(u_i)$ that are at different distances from $x$. Since
the vertices of $N(u_i)$ are pairwise at distance 2, we have
$d(x,u_i) = \min \{d(x,u) : u \in N(u_i)\} +1 = \max \{d(x,u) : u \in
N(u_i)\} -1$. Consequently, independently of the dimension of $C_i$
and the position of $x'$ in $C_i$, the distance $d(x,u_i)$ can be
computed via the following formula:

\begin{lemma}\label{lem-distui}
  For any $x \in S_{i-1}$,
  $d(x,u_i) = \max \{d(x,u) : u \in N(u_i)\} -1$.
\end{lemma}

\begin{figure}[h]
\begin{center}
\includegraphics{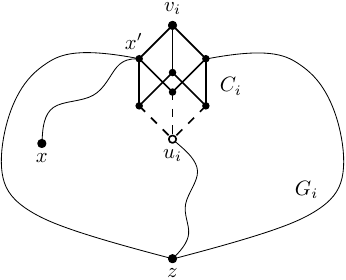}
\end{center}
\caption{Illustration of step $i$ of the reconstruction algorithm}\label{fig-algo}
\end{figure}

We now describe step $i$ of the algorithm. First we define the set
$S_{i-1}$.  The reconstructor picks a vertex $v_i$ in $S_i$ maximizing
$d(z,v_i)$. If $z =v_i$, then the algorithm stops and return
$\Gamma_i$. So, suppose that $v_i\neq z$.  The reconstructor computes
the set $L_i$ of neighbors of $v_i$ in $S_i$ using the distance matrix
$D_i$. If $|L_i| \geq 2$, then by iteratively applying the quadrangle
condition, it identifies all vertices of the cube $C_i$ except
possibly $u_i$. Let $N_i$ be the set of neighbors of $u_i$ in $C_i$.
The reconstructor can detect if the vertices of $N_i$ have a common
neighbor in $S_i$ closer to $z$. If yes, then this vertex is
necessarily $u_i$, otherwise $u_i$ does not belong to $S_i$. If
$|L_i| \leq 1$, then $\dim(C_i)=1$, and $u_i \in S_i$ if and only if
$|L_i| = 1$, in which case, $L_i = \{u_i\}$. In this case, let
$N_i = \{v_i\}$ (as in the previous case, $N_i$ is the set of
neighbors of $u_i$ in $C_i$).  If $u_i \in S_i$, then let
$S_{i-1} = S_i \setminus \{v_i\}$ and if $u_i \notin S_i$, then
$S_{i-1}$ is obtained by adding a new vertex $u_i$ to
$S_i \setminus \{v_i\}$.

The reconstructor then derives the distance matrix $D_{i-1}$ of
$S_{i-1}$ from $D_i$. It first removes the row and the column of $D_i$ corresponding to
$v_i$.
Additionally, if $u_i \notin S_i$, it adds a row and a column for
$u_i$ where
$D_{i-1}[u_i,x] = D_{i-1}[x,u_i]= \max \{D_i[u,x]: u \in N_i\} -1$,
for each $x \in S_i\setminus \{v_i\}$.

Finally, the graph $\Gamma_{i-1}$ is $\Gamma_i$ if $u_i \in
S_i$. Otherwise, $\Gamma_{i-1}$ is obtained from $\Gamma_i$ by adding
the vertex $u_i$, all edges in
$\{{u_i}w : w \in S_{i-1} \text{ and } d(u_i,w) = 1\}$.  If
$\dim(C_i) = 1$, then the edge ${u_i}{v_i}$ is also added to
$\Gamma_{i-1}$.
This ends the description of step $i$ of the algorithm: from the set
$S_{i}$, the distance matrix $D_{i}$, and the graph $\Gamma_{i}$, the
reconstructor has computed $S_{i-1}$, $D_{i-1}$, and $\Gamma_{i-1}$.

Now we prove that the invariants (\ref{inv-1}), (\ref{inv-2}),
(\ref{inv-3}), (\ref{inv-4}) are satisfied after step $i$.
Invariants (\ref{inv-1}) and (\ref{inv-2}) follow from Lemma~\ref{furthest-corner} and the definition of $v_i$. Invariant
(\ref{inv-3}) follows from the definition of $u_i$ and
Lemmas~\ref{boundary_corner}
and~\ref{boundary_corner2}. Since $S_i = S(G_i)$, and by the definitions of $v_i$ and $u_i$, we
have
$S_{i-1} = S_i \setminus \{v_i\} \cup \{u_i\} = S(G_i) \setminus
\{v_i\} \cup \{u_i\} = S(G_{i-1})$. By Lemma~\ref{lem-distui}, the
distances from $u_i$ to all vertices of $S_{i-1}$ have been correctly
computed at step $i$, and thus, by induction hypothesis, $D_{i-1}$ is
the distance matrix of $S_{i-1}=S(G_{i-1})$. If $u_i \in S_i$, then
$\Gamma_{i-1} = \Gamma_i = G[\bigcup_{n \ge j \ge i} S_j] =
G[\bigcup_{n \ge j \ge i-1} S_j]$.  If $u_i \notin S_i$, then
$V(\Gamma_{i-1}) = V(\Gamma_i) \cup \{u_i\} = \bigcup_{n \geq j \geq
  i-1} S_j$. Moreover, for each neighbor $w \in S_{i-1} \cup \{v_i\}$
of $u_i$ in $G$, the edge $wu_i$ is in $E(\Gamma_{i-1})$. If $u_i$ has
another neighbor $w$ in $G$ belonging to $V(\Gamma_{i-1})$, then
$w = v_j$ with $j > i$. However, since $u_i \notin S_j$, this implies
by invariant (\ref{inv-3}) that $u_i \in S_{j-1}$ and thus
$u_i\in S_i$, a contradiction. Therefore, $\Gamma_{i-1}$ is the
subgraph of $G$ induced by $\bigcup_{n \ge j \ge i-1} S_j$. This
establishes invariant (\ref{inv-4}).

\begin{lemma}\label{Gamma0}
  The graph $\Gamma_0$ returned by the reconstructor is isomorphic to $G$. \end{lemma}

\begin{proof}
  By invariant (\ref{inv-1}) of the algorithm, $z$ is the last vertex removed
  from $S$.  By Lemma~\ref{furthest-corner}, when $z$ is considered by
  the algorithm, all vertices of $G$ have been already processed. This
  implies that each vertex $x \in V(G)$ belongs to some $S_i$ and
  thus to $V(\Gamma_0)$, establishing $V(\Gamma_0) = V(G)$. By
  invariant (\ref{inv-4}), $\Gamma_0$ is an induced subgraph of $G$ and is thus
  isomorphic to $G$.
\end{proof}

Lemma~\ref{Gamma0} and Theorem~\ref{CAT(0)=median1} imply the main result of the paper:

\thmCATBoundaryRigid*

\begin{remark}
  The key property of our algorithm is that the vertex $v_i$ chosen at
  step $i$ is a corner of the graph $G_i$, as established in
  Lemma~\ref{furthest-corner}. Note that this lemma holds for any set
  of vertices containing $\partial G_i$, in particular for any set
  containing $S_i$.

  Suppose that instead of starting with $S_n = \partial G$, we start
  the algorithm with a set of vertices $S'_n$ containing $S_n$, the
  distance matrix $D'_n$ of $S'_n$ in $G$, and the subgraph
  $\Gamma'_n$ of $G$ induced by $S'_n$.  Then executing precisely the
  same algorithm and considering the intermediate values $S'_i$, $D'_i$,
  and $\Gamma'_i$ of $S$, $D$, and $\Gamma$,  one can show by
  induction on $i$ that $S_i \subseteq S'_i$ holds for any step $i$.
  Thus our algorithm reconstructs $G$ from $S'_n$ and $D'_n$.

  In particular, since our definition of combinatorial boundary is
  weaker than the one of Haslegrave et al.~\cite{HaScTaTa}, our
  algorithm also reconstruct $G(X)$ starting from the combinatorial boundary 
  of $X$ defined as in~\cite{HaScTaTa}.
\end{remark}

\section{Final remarks}

In this note, we proved that finite CAT(0) cube complexes are boundary
rigid.  This generalizes the results of Haslegrave, Scott, Tamitegama,
and Tan~\cite{HaScTaTa} and settles their main Conjecture 20.  Our
method uses the fact that any ordering of vertices of the 1-skeleton
$G(X)$ of a CAT(0) cube complex $X$ by decreasing the distances to a
basepoint $z$ is a corner peeling.

In our approach, it was important to assume that $G(X)$ is a median
graph, i.e., that $X$ is a CAT(0) cube complex.  For example, consider
the square complex consisting of three squares with a common vertex
$v_0$ and pairwise intersecting in edges incident to this vertex. The
1-skeleton of this complex is the graph $Q^-_3$ (the 3-cube minus one
vertex).  This graph is the 6-cycle $C=(v_1,v_2,v_3,v_4,v_5,v_6)$ and
the vertex $v_0$ incident to three vertices of $C$. There are two ways
to connect $v_0$ to $C$, either with $v_1,v_3,v_5$ or with
$v_2,v_4,v_6$, leading to two isomorphic square complexes $X'$ and
$X''$ with respective $1$-skeletons $G'$ and $G''$.  The graph $G'$
has $v_2,v_4,v_6$ as corners and the graph $G''$ has $v_1,v_2,v_3$ as
corners. Both $G'$ and $G''$ have the cycle $C$ as the boundary, 
furthermore there is an isometry between the two boundaries $\partial G'$ 
and $\partial G''$. However this isometry cannot be extended to an isomorphism 
between the graphs $G'$ and $G''$ and between the complexes $X'$ and $X''$.  
Therefore, any class of finite cube complexes containing the square complex 
$X'\backsimeq X''$ as a member  
is not boundary rigid. This example also shows that  without any additional information we
cannot decide which vertices of $C$ are corners of $G$.
The graph $Q_3^-$ admits a monotone corner peeling if we order the
vertices of $Q_3^-$ by distance to a vertex $z$ of degree
$3$. Consequently, the existence of monotone corner peelings is not
sufficient to ensure boundary rigidity if we are not able to identify
the corners.
By Proposition~\ref{corner-peeling-median}, in median graphs, any
ordering of the vertices by distance to an arbitrary vertex $z$ gives
a corner peeling. Thus, we can choose an arbitrary vertex on the
boundary as a basepoint.

In order to prove that our algorithm effectively
reconstructs the $1$-skeleton $G$ of a cube complex $X$ from its
boundary, it is sufficient to establish the following properties:
\begin{enumerate}[label=(\alph*)]
\item the cubes of $G$ are gated;
\item one can find a basepoint $z\in \partial G$ such that any ordering of the
  vertices of $G$ by distance to $z$ gives a corner peeling.
\end{enumerate}

Condition (a) holds for all partial cubes, i.e., graphs which can be
isometrically embedded into hypercubes. In~\cite{ChChMoWa}, Chalopin,
Chepoi, Moran, and Warmuth proved that every conditional antimatroid
$G$ (conditional antimatroids form a superclass of median graphs and
convex geometries and a subclass of partial cubes) admits a monotone
corner peeling with respect to a specific basepoint $z$, corresponding
to the empty set in the set-theoretical encoding of $G$. However,
conditional antimatroids do not always satisfy condition (b). Indeed,
$Q_3^-$ is a conditional antimatroid, and thus the class of
conditional antimatroids is not boundary rigid.  As explained above,
even if there exists a monotone corner peeling of $Q_3^-$ with respect
to a basepoint on $\partial Q_3^-$, one cannot identify it by knowing
only the boundary of $Q_3^-$. Moreover, there also exist conditional
antimatroids where the specific basepoint $z$ does not belong to
$\partial G$.  However, if, in addition to the distance matrix of
$\partial G$, the specific basepoint $z$ of a conditional antimatroid
$G$ is given as well as the distances from $z$ to all vertices of
$\partial G$, then, proceeding as in the case of median graphs, we can
reconstruct $X_{cube}(G)$.

Haslegrave~\cite{Ha} presented a plane CAT(0) triangle-square complex,
that cannot be reconstructed from its boundary distance matrix.
Haslegrave et al.~\cite{HaScTaTa} asked for conditions under which a
class of finite simplicial complexes is boundary rigid. It will be
interesting to investigate this question for clique complexes of
bridged and Helly graphs, which are well-studied in metric graph
theory~\cite{BaCh_survey, CCHO, Ch_CAT}.  Groups acting on them have
recently been extensively studied in geometric group theory. These
simplicial complexes are not CAT(0) but are considered to have
combinatorial nonpositive curvature; see the papers by Januszkiewicz
and \'Swi{\k{a}}tkowski~\cite{JaSw} and Chalopin,
Chepoi, Hirai, and Osajda~\cite{CCHO}.

\subsection*{Acknowledgements} We are grateful to the referees for a careful reading of the first version and numerous useful comments. 
J.C. was partially supported by ANR project DUCAT (ANR-20-CE48-0006).

\end{document}